\documentclass[12pt,leqno,amsmath]{article}
 \usepackage{amssymb}
 \usepackage{amstext}
 \usepackage{amscd}
 \usepackage{amsthm}
  \usepackage{amsopn}
 \usepackage[dvips]{graphicx}
\def \cfd {\qed}
\newtheorem{Theo}{Theorem}[section]
\newtheorem{Lem}[Theo]{Lemma}
\newtheorem{Prop}[Theo]{Proposition}

\theoremstyle{remark}
\newtheorem*{rem}{Remark}

\theoremstyle{definition}
\newtheorem*{defn}{Definition}

\DeclareMathOperator{\length}{length}
\DeclareMathOperator{\area}{area}
\title{
Functors and Computations in Floer homology with Applications \\ Part II}
\author{C. VITERBO \thanks{D\'epartement de Mah\'ematiques, B\^atiment 425, 
Universit\'e de Paris-Sud, 91405 Orsay Cedex, FRANCE. Supported by C.N.R.S.
U.R.A 1169 and Institut Universitaire de France. Current address: DMA, \'Ecole normale sup\'erieure, PSL University, 45 rue d'Ulm, 75005 Paris, FRANCE.}}
\date{October 15, 1996, \\{\small (revised on \today)}}

\begin{document}

\graphicspath{{./}{./Figures.eps/}{./Part.II/}{/..}{.}}
\maketitle

\begin{abstract}
The results in this paper concern
computations of Floer cohomology using generating functions. The first part
 proves the isomorphism between Floer cohomology and Generating function cohomology introduced by Lisa
Traynor. The second part proves that the Floer cohomology of the cotangent bundle (in the sense of
Part I), is isomorphic to the cohomology of the loop space of the base. This has many
consequences, some of which were given in Part I, others will be given in 
forthcoming papers. 
The results in this paper had been announced (with indications of proof)
 in a talk at the ICM 94 in ZŸrich.
\end{abstract}

\nocite{*}
 \tableofcontents
\vskip 1in

\section {Introduction}

This paper is concerned with computations of Floer cohomology using generating functions.
The first part proves the isomorphism between Floer cohomology and
Generating function cohomology
introduced by Lisa Traynor in \cite{Tr}.
The statement of this theorem was given in \cite{V-ICM94} with rather precise 
indications of proof.
However, since then, we found what we consider a simpler, even though less
natural, proof. A proof along the original indications is given  by 
Milinkovi\'c and Oh in
\cite{Mil-Oh1}, \cite{Mil-Oh2}. 
The second part proves that the Floer cohomology of the cotangent bundle
(in the sense of Part I), is isomorphic to the cohomology of the loop space
of the base. This has many consequences, some of which were given in Part I,
others will be given in forthcoming papers.

We would like to point out a very interesting attempt by Joachim Weber to prove  the main theorem 
 of section \ref{II.2} using a different approach, namely by considering the gradient 
flow of the geodesic energy as a singular  perturbation of the Floer 
flow (see \cite{Weber}).

\section{Floer cohomology is isomorphic to GF-homology} \label{II.1}
Let $L$ be a Lagrange submanifold in $T^*N$,  and
assume $L$ has a generating function quadratic at infinity, that is 
$$
L = \lbrace (x, \frac{\partial S}{\partial x}) \mid \frac{\partial S}{\partial \xi}(x, \xi)=0 \rbrace
$$
where $S$ is a smooth function on $N \times {\Bbb{R}}^k$, such that $S(x, \xi)$
coincides outside a compact set with a nondegenerate quadratic form in the
fibers, $Q(\xi)$. In particular if $N$ is non-compact, $L$ coincides with
the zero section outside a compact set. 

As proved by Laudenbach and Sikorav (see \cite{L-S}) this is the case for 
$L = \phi_1(O_N)$, where $\phi_t$ is a Hamiltonian flow.

Moreover $S$ is "essentially unique" up to addition of a quadratic form in new
variables and conjugation by a fiber preserving diffeomorphism (the "fibers"
are those of the projection $N \times {\Bbb{R}}^k \to N$) see \cite{V-STGGF} and
\cite{Th}.

We may then consider, for $L_0$ and $L_1$ generated by $S_0, S_1$,
the Floer cohomology
$FH^*(L_0, L_1;a,b)$ of the chain complex $C^*(L_0,L_1;a,b)$ generated by
the points, $x$, in
$L_0 \cap L_1$ with $A(x) \in \lbrack a,b \rbrack$ where $A(\gamma)$ is
$\int_{\lbrack 0,1 \rbrack^2} u^* \omega$ where $u:\lbrack 0,1 \rbrack^2 \to
T^*N$ is a map as in figure \ref{figure8} such that
\begin {eqnarray*}
u(0,t) = \gamma_0(t)\; &\; u(s,0) \in L_0\\
u(1,t)=\gamma(t) \; &\; u(s,1) \in L_1
\end{eqnarray*}
where $\gamma_0$ is a fixed path connecting $L_0$ to $L_1$,
and $x$ is identified with the constant path $\gamma_{x}$, at $x$.

Denoting by  $C^k$  the subvector space generated by the intersection
points of
$L$ with the zero section having Conley-Zehnder index $k$, we have a
differential: $$
\delta:C^k(L_0, L_1;a,b) \to C^{k+1} (L_0, L_1;a,b)
$$
is obtained by counting the number of holomorphic strips, that are
solutions of $\bar\partial u = \frac{\partial u}{\partial s}+
J\frac{\partial u}{\partial t}=0$ where $J$ is an almost complex
structure compatible with the symplectic form\footnote{i.e.
$\omega(\xi,J\xi)$ defines a Riemannian metric}
\begin{eqnarray*}
u:{\Bbb{R}} \times \lbrack 0,1 \rbrack \to &T^*N\\
u(s,0) \in L_0  &  u(s,1) \in L_1\\
\lim_{\scriptstyle  S \to \pm \infty} u(s,t)=x_\pm.
\end{eqnarray*}
On the other hand, we have the much simpler space $H^*(S^b, S^a)$, where
$S(x,\xi,\eta)=S_1(x,\xi)-S_0(x,\eta)$, and $S^ \lambda =\{(x,\xi,\eta) \mid
S(x,\xi,\eta) \leq \lambda \}$. This cohomology group does not depend on the
choice of
$S$ according to
\cite{V-STGGF} and
\cite{Th}, up to a shift in index, and was used in \cite{Tr} as a substitute for
Floer cohomology under the name $GF$ cohomology (denoted $GF^*(L_0, L_1;a,b)$).

Our first claim is 
\begin{Theo}
$$
FH^*(L_0,L_1;a,b) \simeq GF^*(L_0,L_1;a,b)
$$ 
\end{Theo}
The proof will take up the rest of this section.
\begin{rem} : This was announced in \cite{V-ICM94}, together with a sketch of
the proof. Our present proof is actually simpler than the one we had in mind,
in particular as far as the applications to the next section are concerned.
\end{rem}

We are first going to introduce a functional interpolating between $A$ and $S$.

Let $H(t,z,\xi)$ be a smooth function on ${\Bbb{R}} \times T^*N \times
{\Bbb{R}}^k$, equal to some non degenerate quadratic form $Q(\xi)$ outside a
compact set.

For $\gamma$ a path between $L_0$ and $L_1$, set $A_H(\gamma, \xi)= A
(\gamma)-\int_0^1 H(t, \gamma(t), \xi)dt$.

Set ${\cal P}(L_0, L_1)$ be the set of paths, $\gamma$, such that 
$\gamma (0) \in L_0 ,\quad
\gamma(1) \in L_1$. Then the critical points of $A_H$ on ${\cal P} (L_0, L_1)
\times {\Bbb{R}}^k$ are the pairs $(\gamma, \xi)$ with 
$$
\left \lbrace
\begin{array}{l}
\stackrel{.}{\gamma} - X_H (\gamma)=0\\
\int_0^1 \frac{\partial}{\partial \xi} H(t, \gamma(t),\xi) dt=0
\end{array}
\right.
$$

where $X_H$ is the Hamiltonian vector field associated to the function $z \to
H(t,z,\xi)$ (i.e.  $\xi$ is "frozen").

\begin{figure}
\centering
	\includegraphics*[scale=.65]{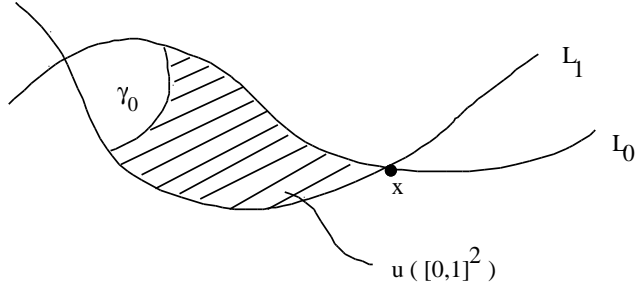}
	\caption{The map $u$.}
\label{figure8}
\end{figure}

If we denote by $\phi ^t_\xi$ the Hamiltonian flow of $X_H$ (for fixed $\xi \in
{\Bbb{R}}^k$) we have $\gamma(t)= \phi^t _\xi(\gamma(0))$, so that the
critical points of $A_H$ correspond to 
$$
\lbrace z \in L_0 \cap (\phi ^t_\xi)^{-1}(L_1) \vert \int_0^1
\frac{\partial}{\partial \xi} H(t, \phi^t _\xi (z), \xi)dt =0\rbrace 
$$
Note that
$$
L_H= \lbrace \phi^1 _\xi (z) \vert z \in L_0, \int_0^1 \frac{\partial}{\partial 
\xi} H (t, \phi ^t_\xi(z), \xi)dt=0 \rbrace
$$
is an immersed Lagrange submanifold, and the above critical points correspond to
$L_H\cap L_0$.

Consider an almost complex structure on $M$ compatible with the 
symplectic form, and denote by $\overline \partial = 
\frac{\partial}{\partial s}+ J \frac{\partial}{\partial t}$.

We now define the Floer cohomology of $A_H$ as usual:

\noindent - $C^*$ will be generated by the critical points of $A_H$ on ${\cal
P}(L_0,L_1)$

\noindent - $<d(x_-,\xi_-),(x_+,\xi_+)>$ equals the algebraic number of
solutions of 
$$
\left \lbrace
\begin{array}{l} u: {\mathbb  R}\times [0,1] \Leftrightarrow M \\
\overline \partial u(s,t)=-\nabla H(t,u(s,t),\xi(s))\\
\frac{d}{ds}\xi(s)=-\int_0^1 \frac{\partial}{\partial \xi} H(t,u(s,t),\xi(s))dt
\end{array}
\right.
$$
satisfying $\displaystyle \lim_{S \to\pm\infty}(u(s,t), \xi(s))=(x \pm, \xi
\pm)$.

Note that the set of such solutions has its image in a bounded subset of $T^*N
\times {\mathbb R} ^k$.

Indeed, for $u$ outside a bounded subset, $H$ vanishes, so the equation becomes

$$
\left \lbrace
\begin{array}{l}
\overline \partial u=0\\
\frac{d}{ds}\xi(s)= - \int_{0}^1 \frac{\partial}{\partial \xi}H(t,0,\xi) dt
\end{array}
\right.
$$
But the first equation cannot hold in the region foliated by pseudoconvex
hypersurfaces.

If on the other hand $\vert \xi \vert$ is large the equation will be
$$
\left \lbrace
\begin{array}{l}
\overline \partial u=-\nabla H(t,u,\xi)\\
\frac{d}{ds}\xi(s)= B \xi
\end{array}
\right.
$$
where $Q(\xi)=<B \xi, \xi>$. 

But the second equation is such that any bounded
set is contained in a set with the property that if a trajectory exits the set, 
it will never reenter it.

Thus, the set of bounded solutions has its image in a bounded set.
Then the set of solutions  satisfies the same formal properties
as the set of solutions of the usual Floer equation. The proofs are
just verbatim translations of those in \cite{F1}.
 
\begin{defn} The cohomology of $(C^*(L_0, L_1,H),d)$ 
will be denoted by $FH^*(L_0, L_1,H)$.
If we restrict the complex to those solutions with $A_H(\gamma ,\xi) \in [a,b]$
the cohomology is denoted  $FH^*(L_0, L_1,H; a,b)$.
\end{defn}

A first result is
\begin{Lem}\label{2.2} Let $\psi_\lambda$ be a Hamiltonian flow on 
$T^*N$, and $H_\lambda$ , $F_\lambda$ be the 
Hamiltonians with  flows $\psi_\lambda^{-1} \circ \phi_\xi^t \circ \psi_\lambda$
and  $\phi_\xi^t \circ (\psi_\lambda^t )^{-1}$.

Then

$$
FH^*(L_0,
\psi_1(L_1),H_0)=FH^*(\psi_1^{-1}(L_0),L_1,H_1)=FH^*(L_0,L_1,F_1)
$$
\end{Lem}

\begin{proof}
Consider first the chain complex
$C^*(\lambda) = C^*( \psi_\lambda^{-1}(L_0),\psi_\lambda^{-1}
\psi_1(L_1),H_\lambda)$ for $\lambda$ in $\lbrack 0,1 \rbrack$. It is
generated by points in $(\psi_\lambda^{-1}(L_0))_{H_{\lambda}} \cap
\psi_\lambda^{-1}\psi_1(L_1)$.

But $(\psi_\lambda^{-1}(L_0))_{H_{\lambda}}$ is defined as
$$
\lbrace \phi_{\lambda,\xi}^1(\psi_\lambda^{-1}(z)) \vert z \in L_0;\int_0^1 \partial_\xi
H_\lambda(t,\phi_{\lambda,\xi}^t(\psi_\lambda^{-1}(z)),\xi)dt=0 \rbrace
$$
and since $H_\lambda(t,u,\xi)=H(t,\psi_\lambda(u),\xi)$ we have
$H_\lambda(t,\phi_{\lambda,\xi}^t \circ \psi_\lambda^{-1}(z), \xi)=
H(t,\phi_\xi^t(z),\xi)$ and $\phi_{\lambda, \xi}^t \circ
\psi_\lambda^{-1} = \psi_\lambda^{-1} \circ \phi_\xi^t$, we have that
$(\psi_\lambda^{-1} (L_0))_{H_{\lambda}} = \psi_\lambda^{-1}
((L_0)_{H})$ and $\psi_\lambda^{-1}((L_0)_H) \cap
\psi_\lambda^{-1}(\psi_1(L_1))= \psi_\lambda^{-1}((L_0)_H \cap
\psi_1(L_1))$

Thus the generators of $C^*(\lambda)$ do not depend on $\lambda$.

Note also that the intersection points of
$(\psi_\lambda^{-1}(L_0))_{H_{\lambda}}$ and
$\psi_\lambda^{-1}(\psi_1(L_1))$ stay transverse provided they are
transverse for $\lambda=0$ (that we shall always assume), and this is sufficient,
together with the fact that the value of $A_{H_\lambda}$ on the critical points
does not depend on $\lambda$,
to imply that the cohomology of $C(\lambda)$ will not depend on
$\lambda$.

Consider now the second equality. Note that we may deform $H$, provided
its time one flow is unchanged, and the same for the flow $\psi_\lambda$.

We may thus assume that  $\partial_z H$ vanishes,
(remember that  $H$ must be quadratic in $\xi$, hence it cannot vanish) for $t$ in $\lbrack 0, \frac{2}{3}
\lbrack$, and $K_\lambda$, vanishes for $t$ in $\lbrack \frac{1}{3},1
\rbrack$.

Then the flow $(\psi_\lambda^t)^{-1} \circ \phi_\xi^t$ is generated by
$H(t, \psi_\lambda^t(z), \xi) + K_\lambda(t,z)=F_\lambda(t,z,\xi)$ and
therefore
$$
\partial_\xi F_\lambda (t,z,\xi)= \partial_\xi H(t, \psi_\lambda^t(z), \xi)
$$
As a result,
\begin{eqnarray*}
(\psi_\lambda^{-1}(L_0))_{H_{\lambda}} &=& \lbrace \phi_{\lambda, \xi}^
{1}(\psi_\lambda^{-1}(z)) \vert z \in L_0, \int_0^1 \partial_\xi
H_\lambda(t,\phi_{\lambda, \xi}^t \circ \psi_\lambda^{-1}(z), \xi)dt=0
\rbrace\\
&=&\lbrace \psi_\lambda^{-1} \phi_\xi^1(z) / z \in L_0; \int_0^1
\partial_\xi H(t, \phi_\xi^t(z),\xi)dt=0\rbrace\\
&=& \lbrace \psi_\lambda^{-1} \phi_\xi^1 (z) / z \in L_0; \int_0^1
\partial_\xi F_\lambda(t,(\psi_\lambda^t)^{-1} \circ \phi_\xi^t(z),
\xi)dt=0\rbrace\\
&=&L_{F_{\lambda}}
\end{eqnarray*}

Thus we have 
$$
(\psi_\lambda^{-1}(L_0))_{H_{\lambda}}=\psi_\lambda^{-1}((L_0)_H)=L_{F_{\lambda}}
$$
and $L_{F_{\lambda}} \cap \psi_\lambda^{-1} \psi_1(L_1)$ is independent
of $\lambda$, and by the same argument as above,  $FH^*(L_0, \psi_\lambda^{-1}
\psi_1(L_1),F_\lambda)$ does not depend on $\lambda$. 
Using again the invariance of the critical levels, we  proved the
second equality.
\end{proof}

In particular for $L_0=O_N,L=\phi_1(O_N)$ 

\noindent  we get
$$
FH^*(O_N,L,0)=FH^*(O_N, O_N,H)
$$
Now let $S(q,\xi)$ be a $g.f.q.i.$ for $L$. The flow of $X_S$ is given by
$$
\left \lbrace
\begin{array}{l}
\dot q=0\\
\dot p =\frac{\partial S}{\partial q}(q,\xi)
\end{array}
\right.
 {\rm{ or
}}
\left \lbrace
\begin{array}{l}
q(t)\equiv q(0)\\
p(t)=t \frac{\partial}{\partial q} S(q,\xi) + p(0).
\end{array}
\right.
$$
and $(O_N)_S= \lbrace q,\frac{\partial S}{\partial q}(q, \xi)) \vert \int_0^1
\partial_\xi S(q(t), \xi)dt=0 \rbrace$ but since $q(t) \equiv q(0)$, we
have that $\int_0^1 \partial_\xi S(q(t), \xi)dt= \partial_\xi S(q, \xi)$,
hence $(O_N)_S$ is the Lagrange submanifold generated by $S$.

Now we claim
\begin{Lem} For $S$ a $g.f.q.i$ we have $FH^*(O_N,O_N,S;a,b)
\simeq H^*(S^b,S^a)$.
\end{Lem}

\begin{proof}
Indeed the generators on both sides are critical points of $S$, and
connecting trajectories solve:

$$
\left \lbrace
\begin{array}{l}
\overline \partial u=-\nabla_q S(q(s,t),\xi(s))ds\\
\frac{d}{ds} \xi(s)=-\int_0^1 \partial_\xi S(q(s,t), \xi(s))dt.
\end{array}
\right.
$$
where $u(s,t)=(q(s,t),p(s,t))$

Now the function $f(q,p)=\vert p \vert$ is pluri-subharmonic for our $J$,
and since $\overline \partial u$ is tangent to $\ker df$ (note that $\nabla_q S$
is horizontal), we have that $ f\circ u$  satisfies the maximum principle.
Since $p(s,1)=p(s,0)=0$, we must have $p\equiv 0$, and thus

\begin{eqnarray*}
\frac{\partial}{\partial s} q(s,t) & = & - \nabla_q S(q(s,t), \xi(s))\\
\frac{\partial}{\partial t} q(s,t) & = & 0
\end{eqnarray*}
In other words, $q$ only depends on $s$, and satisfies

$$
\left \lbrace
\begin{array}{l}
\dot q= -\nabla_q S(q(s), \xi(s))\\
\dot \xi (s)  =-\partial_\xi S(q(s), \xi(s))
\end{array}
\right.
$$
and $\displaystyle\lim_{s \to \pm \infty}(q(s), \xi(s))=(q_\pm, \xi_\pm)$
are critical points for $S$.

Thus our connecting trajectories, are just bounded gradient trajectories
for $S$, and thus the coboundary map is the same on both complexes,
hence the cohomologies are the same. \cfd

We may finally conclude our proof. 
Let $L=\phi_1(0_N)$, so that it has a generating function $S$.

Now let $S_{\lambda }$ be the generating function of $L_ \lambda =\phi_
\lambda^{-1}(L)$, so that $S_0=0$ and $S_1=S$. We claim that the modules
$FH^*(L_\lambda,0_N,S_ \lambda ,a,b)$ are all isomorphic. Indeed, $L_ \lambda \cap
(0_N)_{S_ \lambda } = \phi_ \lambda^{-1} (0_N) \cap \phi_ \lambda^{-1} (L) = \phi_
\lambda^{-1} (0_N \cap L) $ is constant and the usual argument proves the constancy
of the Floer cohomology ring. It remains to show that $$FH^*(\phi_
1 ^{-1}(0_N),0_N,0)= FH^*(0_N, \phi_ 1(0_N),0)$$ Again this follows from the equality

$FH^*(\phi_ 1 ^{-1}(0_N),0_N,0)= FH^*(\phi_t\circ \phi_1^{-1}(0_N), \phi_t(0_N),0)$,
because $\phi_t\phi_1^{-1}(0_N) \cap \phi_t(0_N) = \phi_t(\phi_1^{-1}(0_N) \cap
0_N)$.

We thus proved that
$$
FH^*(O_N,L,0;a,b) \simeq FH^*(O_N, O_N,S;a,b)
$$

Since
$$
FH^*(O_N,L,0;a,b) \simeq FH^*(L;a,b)
$$
 and the above lemma proves states that
$$
FH^*(O_N, O_N,S;a,b) \simeq H^*(S^b,S^a)=GH^*(L;a,b)
$$

this concludes our proof. 
\end{proof}

\section{The isomorphism of
\protect \boldmath$FH^*(DT^*N)$ \protect\unboldmath
and  \protect \boldmath$H^*(\Lambda N)$ \protect\unboldmath }\label{II.2}

We will now compute  $FH^*(M)$ for $M=DT^*N = \{(q,p)\in 
T^*N \mid \vert p \vert _{g}=1\}$ for some metric $g$. 

Let us denote by $\Lambda N$ the free lopp space of $N$ (i.e. 
$C^0(S^1, N)$ ).

We have:

\begin{Theo}\label {3.1} 
$$
FH^*(DT^*N) \simeq H^*(\Lambda N)
$$
where $\Lambda N$ is the free loop space of $N$. The same holds for $S^1$
equivariant cohomologies with rational coefficients.
\end{Theo}

The proof will take up the rest of this section. The next two lemmata are
valid in any manifold satisfying the assumptions of section \ref{II.1}.
First let us consider the
diagonal $\Delta$ in $M \times \overline M$, where $\overline M$ is the
manifold $M$ endowed with the symplectic form $-\omega$, and $\varphi^t$ 
is the flow of $H$ on $M$. 

Our first result is

\begin{Lem}\label{3.2}
$$FH^*(H;a,b) \simeq FH^*(\Delta,({\rm{Id}} \times \varphi^1) \Delta;a,b).
$$
\end{Lem}

\begin{proof}
Remember that $FH^*(H,a,b)=FH^*(0,0,H,a,b)$.

We first point out that the cochain spaces associated to both sides are the same,
and generated by the fixed points of $\varphi^1$.

The connecting trajectories are, for the left hand side given by

$$
\overline \partial v=-\nabla H(t,v)
$$
$$
v:{\mathbb R} \times S^1 \rightarrow M
$$
while for the right hand side, we first notice that $FH^*(\Delta,({\rm{id}}
\times \varphi^1) \Delta ;a,b) \simeq FH^*(\Delta,({\rm{Id}} \times
\varphi^1) \Delta, 0;a,b) \simeq FH^*(\Delta, \Delta,0 \oplus H;a,b)$
(Note: by $K=0 \oplus H$ we mean the Hamiltonian on $M \times \overline M$
 defined by $K(t,z_1,z_2)=H(t,z_2$)).

The second isomorphism follows from lemma \ref{2.2}.

Now the coboundary map for this cohomology is obtained by counting
solutions of
$$
\overline \partial u= - \nabla K
$$
where $u=(u_1,u_2) : {\mathbb R}\times S^1 \Leftrightarrow M\times 
\overline M$ and $\overline \partial u_1=0;
\quad \overline \partial u_2=-\nabla H(t,u_2)$.

Moreover, since $t \rightarrow u(s,t)$ connects $\Delta$ to itself, we
have $u_1(s,0)=u_2(s,0)$ and $u_1(s,1)=u_2(s,1)$.

Therefore we may glue together $u_1$ and $u_2$ to obtain a map $\widetilde v :
S^1 \times {\Bbb{R}} \rightarrow M$ such that 
$$
\overline \partial \widetilde v =- \nabla \widetilde H(t,\widetilde v)
$$

\begin{eqnarray*}
{\rm{where}} \quad \widetilde H(t,u) &=& 0 \quad  {\text{for}}\quad  t \quad
{\rm{in}}\quad \lbrack 0,1\rbrack\\
 &=&H(t,u) \quad {\rm{for}} \quad  t\quad  {\text {in}} \quad \lbrack 1,2
\rbrack
\end{eqnarray*}
(the circle being identified with ${\Bbb{R}} / 2{\mathbb Z} )$.
We shall make the simplifying assumption that $H(t,z)=0$ for $t$ close 
to $0$ or $1$, so that $\tilde H$ is continuous in $t$. This is of 
course not a restriction on the time one map. 

But the time $2$ map of $\widetilde H$ coincides with the time $1$ map
for $H$, hence we may continuously deform one equation into the other,
and the two cohomologies are isomorphic.
(Remember that if we have a family of Hamiltonians depending
continuously on some parameter, but having the same time $1$ map, then
the corresponding Floer cohomologies are alos independnt from the 
parameter).

\end{proof}

Assume $\varphi^1$ is equal to $\psi^r$. We denote by $\Gamma_\varphi$
the graph of $\varphi$ (i.e. $\Gamma_{\phi}= ({\rm{id}} \times \varphi) \Delta $), and
$\Gamma_\psi^{\rho,r} = \lbrace (z_1, \psi(z_2),z_2, \psi(z_3) \ldots ,
z_{k-1}, \psi(z_r), z_r, \psi(z_1)) \rbrace$ in $(M \times \overline M)^r$

Set $\Delta^r = \Delta \times \quad \times \Delta $ ($r$ times)

\begin{Lem}\label{3.3}. For $\varphi = \psi^r$ we have
$$
FH^*(\Delta, \Gamma_\varphi;a,b) \simeq FH^*(\Delta^r,
\Gamma_\psi^{\rho,r}; a,b)
$$
\end{Lem}

\begin{proof} Clearly, we may identify $\Gamma_\varphi \cap
\Delta$ and $\Delta^r \cap \Gamma_\psi^{\rho,r}$, since a point in this intersection is given by $(z_1, \psi(z_2),
\ldots z_r, \psi(z_1))= (z_1, z_1, \ldots z_r, z_r)$ that is $z_1=
\psi(z_2); z_2= \psi(z_3), \ldots , z_r=\psi(z_1)$
or else

$$
z_1=\psi^r(z_1),z_i=\psi^{r+1-i}(z_1) \quad{\text{for}}\quad i\geq 2
$$

Note that there is a ${\Bbb{Z}}/r$ action on the set of such points,
induced by $z \rightarrow \psi(z)$ on $\Gamma_\varphi \cap \Delta$ and
by the shift $(z_j) \to (z_{j-1})$ on $\Gamma_\psi^{\rho,r} \cap
\Delta^r$, and these two actions obviously coincide.
We shall not mention this point in the proof, but all our results
hold for  ${\Bbb{Z}}/r$ equivariant cohomology with any coefficient ring, and
this eventually allows us to recover the  $S^1$ equivariant cohomology (with rational 
coefficients), due to the Lemma 
in  appendix 2 of \cite{V-Loopspaces}.

Now let us compare the trajectories.

For the second cohomology $(u_j,v_j):{\Bbb{R}} \times \lbrack 0,1 \rbrack
\rightarrow M$ satisfying
$$
\left \lbrace
\begin{array}{l}
\overline \partial u_j = \overline \partial v_j = 0\\
u_j(s,0) = v_j(s,0) = 0\\
u_j(s,1) = \psi(v_j(s,1))
\end{array}
\right.
$$If we set $w_j(s,t) = \psi_t(v_j(s,t))$ so that 

$$
\left \lbrace
\begin{array}{l}
\overline \partial w_j = \nabla K(s,t) \quad \overline \partial
u_j=0\\ 
u_j(s,0) = w_j(s,0)\\
u_j(s,1) = w_{j+1}(s,1)
\end{array}
\right.
$$
We may now glue together the $u_j$ and $w_j$ as in figure \ref{fig9b} to
get a map
$$
u:{\mathbb{R}} \times {\mathbb R}/2r {\mathbb Z} \rightarrow M
$$
such that 
\begin{eqnarray*}
u(s,t) &=&u_j(s,t-2j) \quad \quad \quad {\text{ for }}\quad 2j \leq t \leq
2j+1\\ u(s,t)&=&w_j(s,t-2j-1)\quad {\text{ for }}\quad 2j+1 \leq t <
2j+2
\end{eqnarray*}
It satisfies $\overline \partial u=\nabla F(u)$ where 
\begin{eqnarray*}
F(t,u) & = & K(t-2j,u) \quad {\text{ for }}\quad 2j\leq t \leq 2j+1\\
& = &0\qquad \qquad {\text{otherwise}}
\end{eqnarray*}

Thus the time $2k$ map of $X_F$ is equal to $\psi^k= \varphi$.

We thus identified the Floer trajectories defining $FH^*(\Delta,
\Gamma_\psi^{\rho,r};a,b)$ with those defining 
$$
FH^*(H;a,b) \simeq FH^*(\Delta,\Gamma_\varphi;a,b)
$$
This concludes our proof.

\end{proof}

\begin{figure}
  	\centering
	\includegraphics*[scale=.8]{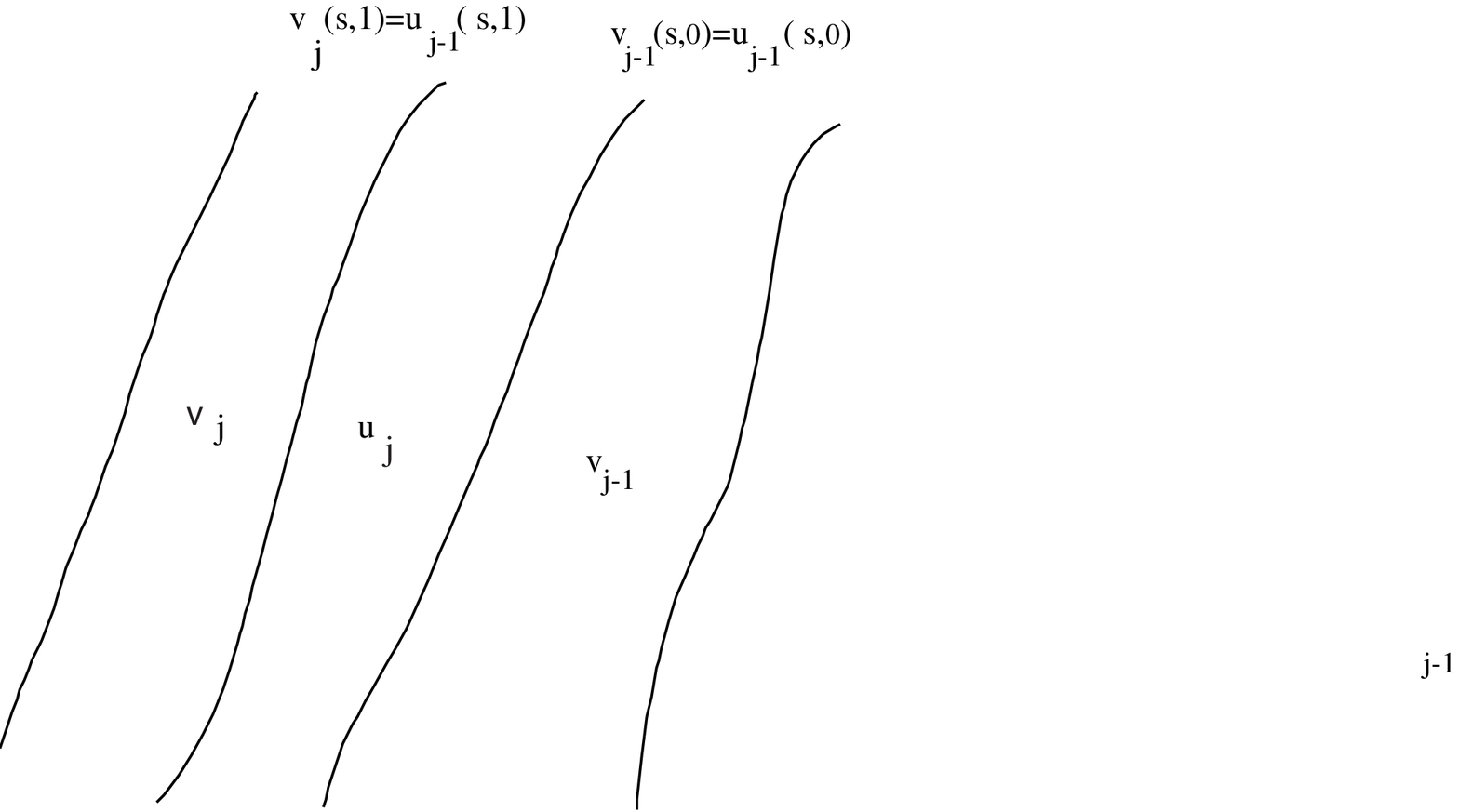}
	\caption{The glueing of the maps $u_{j}$ and $v_{j}$}
 \label{fig9b}
\end{figure}

Let us now prove  theorem \ref{3.1}. Let $H$ be a Hamiltonian
equal to $H(t,q,p)=c \vert p \vert$ for $\vert p \vert$ large, where $c$ is some
constant (that differs from the length of a closed geodesic, but will
eventually become large), and $\varphi^t$ its flow.

For some integer $r$ (that will eventually become large) set $\psi =
\varphi^{1/r}$, so that $\psi^r= \varphi^1= \varphi$.

Then according to lemmata \ref{3.2} and \ref{3.3}, we have that
$$
FH^*(H;a,b)\simeq FH^*(\Delta^r, \Gamma_\psi^{\rho,r};a,b)
$$
Now we may assume that $H(t,q,p)=h(\vert p \vert)$ with $h$ increasing
and convex, and $k$ is the Legendre dual of $h$. Let $q(t)$ be a loop in $M$,
and $E(q)=\int_{S^1} k(\vert {\dot  q} \vert) dt$. Then we
claim
$$
FH^*(H;a,b) \simeq H^*(E^b, E^a)
$$
We know from \cite{V-Loopspaces} that there is a subset ${\cal U}_{r,
\varepsilon}$ in $\Delta^r$, defined by
$$
{\cal U}_{r, \varepsilon} = \lbrace (q_j, P_j)_{j\in \mathbb Z / r \mathbb Z}
\vert d(q_j, q_{j+1}) \leq \varepsilon / 2 \rbrace
$$
for $\varepsilon$ small enough, independent from $r$, such that
$\Gamma_\psi^{\rho,r}$ (that is $\Gamma_\Phi$ in the notation of
\cite{V-Loopspaces}) is the graph of $dS_\Phi$ over ${\cal
U}_{r,\varepsilon}$.

From proposition 1.8 in \cite{V-Loopspaces}, we have that there is a
pseudogradient vector field $\xi_\Phi$ of $S_\Phi$ such that, denoting by
$HI^*$ the cohomological Conley index (see \cite{Conley}), we have, 
for some integer $d$, 
$$
HI^*(S_\Phi^b , S_\Phi^a; \xi_\Phi) \simeq H^{*-d}(E_r^b, E_r^a)
$$
where

$$
E_r^b= \lbrace q=(q_j) \in N^r \vert  d(q_j,q_{j+1}) \leq \varepsilon / 2
\quad {\text{and}}\quad 
 E_r(q)<b \}
$$
(Note: in \cite{V-Loopspaces}, $E_r^b$ is denoted $\Lambda_{r,
\varepsilon}^b$)

$(E_r(q) = \displaystyle {\sup}_ P S_\Phi(q,P)$ and $d$ is some
normalizing constant (equal in fact to $r n$).

On the other hand $E_r^b \simeq E^b$ where $E^b= \lbrace q \in \Lambda M
\vert E(q) \leq b \rbrace$ (see e.g. \cite{Mi}), so we only
need:

\begin{Lem}
$$
FH^*(\Delta^r, \Gamma_\psi^{\rho,r}; a,b) \simeq HI^{*+d}(S_\Phi^b,
S_\Phi^a, \xi_\Phi)
$$
\end{Lem}

\begin{proof} We would like to find an almost complex structure $J=J_0$, on $(T^*N
\times \overline{T^*N})^r$ such that the holomorphic maps corresponding
to Floer trajectories are also in one to one correspondence with bounded
trajectories of $\xi_\Phi$.

We first choose $h$ so that the critical values of $S_\Phi$ are in $\lbrack
- \delta, 0 \rbrack$. Indeed, we may impose that $h'(u)u-h(u)$ is in $\lbrack
- \delta, 0 \rbrack$ (see Figure \ref{fig10b}).

Then the Floer trajectories used to define the left hand side will have
area less than $2\delta$.

 \begin{figure}
	 \centering
	 \includegraphics*[scale=.45]{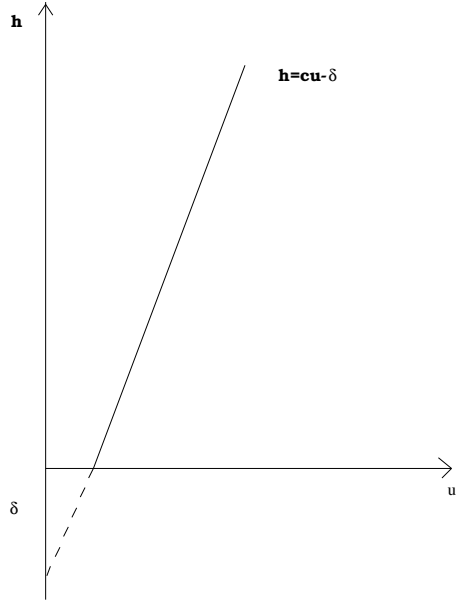}
	 \caption{The hamiltonian $h$}
	 \label{fig10b}
 \end{figure}

Now, let us show that for $\delta$ small enough, such a Floer trajectory
must stay inside $T^* {\cal U}_{r,\varepsilon_{0}}$.

Indeed, since all critical points are inside $T^*({\cal U}_{r,
\varepsilon_{0/2}}) $, we have, for
a Floer trajectory exiting from $T^* {\cal U}_{r, \varepsilon_{0}/2}$, that it
defines a $J$-holomorphic curve $\Sigma$ such that 

\begin{itemize}
\item [(i)] $$
\partial \Sigma\subset \Delta^r \cup\Gamma_\Phi^{\rho,r}.$$
\item[(ii)]
$$\Sigma \quad\text{is in}\quad  T^*({\cal U}_{r,\varepsilon_{0}}- {\cal
U}_{r,\varepsilon_{0/2}}) .
$$
\end{itemize}

Now we are in the following abstract situation:

Let $V_0 \subset V_1$ be convex  sets, $L_0, L_1$ be disjoint Lagrange 
submanifolds in $T^*(V_1 - V_0)$.
We set $S_t = \partial V_t =\{x\in V_1 \mid (1-t) d(x,V_0)= t d(x,V_1)\}$, and
$B_t = \partial T^*V_t = T^*_{S_{t}}V$.

Let  $\Sigma $ be  a holomorphic curve with boundary in $L_0 \cup  L_1$.
Then we claim that the area of $\Sigma$ is bounded from below. 

First let us consider the case of a pseudo-holomorphic curve, $\Sigma$ closed
in $T^*(V_1 - V_0)$ (and in particular with boundary in $T^*_{S_0}V \cup T^*_{S_1}V
$ ), such that $\partial \Sigma \cap T^*_{S_0}V = \partial \Sigma \cap B_0$ is non-empty.

Then $B_{1/2}$ must intersect $\Sigma$, since otherwise, as $\partial \Sigma \cap 
B_0 \neq \emptyset$, there would be an interior tangency between some $B_t$ (for $0<t<1/2$)
and $\Sigma$, and this is impossible by the pseudoconvexity of $B_t$.

Thus for $\alpha=\frac{1}{2} d(V_0,V_1)$, and $x_0 \in B_{1/2} \cap \Sigma $ 
we have that $B(x_0,\alpha)$ is in $T^*(V_1-V_0)$, and thus, we have

$$\area (\Sigma) \geq \area (\Sigma \cap B(x_0,\alpha) 
\geq \pi \alpha ^2 \exp (\eta(c\alpha))$$

where $c$ is an upper bound for the sectional curvatures of the metric 
$g_0$ associated to $\omega$ and $J$ (see Appendix).

Now consider the case where $\Sigma$ has a boundary contained in the 
union of the two Lagrange submanifolds $L_0 $, $L_1$.

Let $U_0, U_1$ be tubular neighbourhoods of $L_0,L_1$ respectively, and assume that
they are disjoint, symmetric (i.e. there is an anti-holomorphic diffeomorphism of
$U_{i}$, fixing $L_i$), and pseudoconvex. This can be easily achieved, through a 
perturbation of $J$ near the $L_i$.

Consider now $\partial \Sigma \cap L_i =\gamma _i$. Then  either $\gamma _0$ and
$\gamma _1$ are both contained inside  $B_{1/2}$ or one of them is not.

In the first case, consider $\Sigma \cap T^*(V_1-V_{1/2})$. Then this intersection
does not have a boundary in $T^*(V_1-V_{1/2})$, except on $B_{1/2} \cup B_1$, and
we thus have again a lower bound on the area of $\Sigma$ as in the case of a
closed curve, except that $\alpha$ is to be replaced by $\alpha /2$.

In the second case, assume for instance that $B_{1/2}$ intersects $\gamma _0$.

Let us then consider the symetrization of $\Sigma$ inside $U_0$, where 
$J$ is integrable, (see for instance \cite{Sibony}). 
This will be a closed curve  $\widehat \Sigma$ in $U_0$, that has a
point $x_0$ in $B_{1/2}$, and we also have a ball $B(x_0,\alpha)$ in $U_0\cap
T^*(V_1-V_0)$ for $\alpha \leq \inf \{ \frac{1}{2} d(V_0,V_1), \frac{1}{2} d(L_0,\partial
U_0)\}$, and again, we get a lower estimate of the area of $\widehat \Sigma$.
Since $\area (\Sigma) \geq {1/2} \area (\widehat\Sigma)$, we also get an estimate on
the area of $\Sigma$.

This proves our abstract statement. We now claim that we are in the above
framework, with $\alpha \simeq \varepsilon _0/4$.

Indeed, the diameter of  $T^*({\cal U} _{r,\varepsilon}-{\cal U}
_{r,\varepsilon/2})$
is $\varepsilon/2$, and we have to show that $\Delta ^r$ has a pseudoconvex
symmetric neighbourhood
of radius $\varepsilon/4$.

But we will show that 
$$\Gamma_\Phi^{\rho,r} \cap T^*({\cal U}_{r,\varepsilon_{0}}- {\cal
U}_{r,\varepsilon_{0}/2}) \cap \lbrace \vert X_j \vert \leq \varepsilon_{0}
/ 4\vert Y_j \vert \leq \varepsilon_{0}/4 \rbrace = \emptyset.\leqno(ii)
$$

This implies that $\lbrace \vert X_j \vert \leq \varepsilon / 4\vert Y_j 
\vert \leq \varepsilon/4 \rbrace $  is a tubular
neighbourhood of the zero section, is disjoint from $\Gamma_\Phi^{\rho,r}$
in the region we are considering. Thus we have our $U_0$, with radius
$\varepsilon/4$. A similar fact would hold for $\Gamma_\Phi^{\rho,r}$.

Let us now prove our last claim.

Indeed we only  have to show that  $\vert  
\frac{\partial  S_\Phi}{\partial P_j}\vert  \geq \varepsilon /4$ for $ (q, P)$ in $ {\cal U}_{r,\varepsilon _{0}} - {\cal 
U}_{r,\varepsilon _{0}/2} .$

But  ${\partial S_{\Phi }\over \partial P_{j}} \simeq  q_{j+1} - 
q_{j} - \frac{1}{ r} \frac{\partial H}{\partial p} + \eta (q_{j+1} - 
q_{j}, P_{j})$ (see \cite{V-Loopspaces}, proof of 1.2)   where  $\eta 
(0,P) = d\eta (0,P) = 0$, and since for some  $j$,  $d(q_{j+1}, q_{j}) 
\geq  \frac{\varepsilon}{ 2}$, we have
 $\vert  X_{j}\vert  = \vert \frac{\partial S}{\partial 
P_{j}}\vert  \geq  \frac{\varepsilon} {2} - \frac{C}{r} \geq  
\frac{\varepsilon _{0}}{4}$ for $ r$ large enough.

It is particularly important to notice that our lower bounds are independent from
$r$, since they only depend on an upper bound for the sectional curvatures
of the metric, and this quantity stays bounded as $r$ goes to infinity.
\end{proof}

If we choose   $\delta  < \frac{1}{4}  \varepsilon $ we get that 
all Floer trajectories for  $J_{0}$ must stay in  
$T^{*} {\cal U}_{r,\varepsilon _{0}}$.

	We claim that

(i) for a suitable choice of  $J_{1}$, the  $J_{1}$ holomorphic 
curves defining the Floer cohomology are in one to one correspondence 
with trajectories of  $S _{\Phi }$.

(ii) there is a family  $J_{\lambda }$ of almost complex structures 
connecting  $J_{0}$ to  $J_{1}$, taming  $\omega $ and making  
$T^{*}_{\partial {\cal U}_{r,\varepsilon _{0}}} {\cal 
U}_{r,\varepsilon _{0}}$ pseudo convex.   
 Note that our almost complex structures will be time dependent, but this
is not important.

\begin{Lem}

	Let  $U$ be a manifold with boundary, and  $f : U \to {\Bbb{R}}$ 
be a smooth function. Then for  $L$ =graph$(df)$,  $\phi_{t}$ 
the Hamiltonian flow  $(q,p) \to  (q,p+t df(q))$,
	set  $J(t,u) = (\phi_{t})_{*} J_{0}(u)$.

	Then, there is a one to one correspondence between solutions of
 
$$
\left \lbrace
\begin{array}{l}
\dot q(t) = - \nabla  f(q(t))\\
\displaystyle\lim_{s \to \pm \infty}q(t) = x_{\pm }\end{array}
\right.
$$

and

$$
\left \lbrace
\begin{array}{l}
\overline \partial _{J} v = 0\\
      v(s,0) \in  0_{N}\quad \quad \quad , \quad \quad
  v(s,1) \in  L\\
\displaystyle\lim_{s \to \pm \infty} v(s,t) = x_{\pm }
\end{array}
\right.
$$
\end{Lem}
\begin{proof} 
	We have, setting  $v(s,t) = \phi_{t} u(s,t)$

\begin{eqnarray*}
\overline \partial _{J} v = d \phi_{t}(u) \frac{\partial}{
\partial s} u (s,t)\\
+ J(\frac{\partial}{ \partial t} \phi_{t}) (u) \\
= d \phi_{t}(u) \lbrack \frac{\partial}{\partial t} +d \phi_{t}(u)
\lbrack\frac{\partial}{ \partial S} +  d\phi^{-1}_{t}( \phi_{t}(u)) J
d\phi_{t}(u) \frac {\partial }{\partial  t} u\\
 + \nabla H(u)\rbrack\\
=d\phi_{t}(u)\lbrack \overline \partial _{J_{0}}u + \nabla H (u) \\
= d\phi_{t}(u) \lbrack \overline \partial _{J_{0}}u + \nabla 
H(u)\rbrack .
\end{eqnarray*}

Thus

$$
\left \lbrace
\begin{array}{l}
\overline \partial _{J_{0}} u = - \nabla H(u)\\
 u(S,0) \in  O_{U} \quad \quad \quad \quad 
 u(s,1) \in  O_{U} .
\end{array}
\right.
$$

Now since  $H(q,p) = f(q)$,  $\nabla H(q,p) = \nabla  f(q)$, and in 
local coordinates, we have  $dp_{j} \cdot  \nabla  f(q) = 0$.

Hence  $d(\vert p\vert) \cdot  \overline \partial _{J_{0}} u = 
\frac{1}{ \vert p \vert} \displaystyle \sum ^{n}_{j=1} dp_{j} \cdot  
\nabla  f(q) = 0$ and  $\vert p \circ  u\vert$ satisfies the maximum 
principle. But since  $u(s,0) \in  O_{U}$,  $u(s,1) \in  O_{U}$, we 
have  $p \circ  u \equiv  0$, hence  $u(s,t) = q(s,t),0)$. Now  
$\frac{\partial q}{ \partial t} + \frac{\partial p}{\partial S} = 0$ is 
the second half of  $\overline \partial _{J_{0}} u = - \nabla 
H(u)$, hence  $q(s,t) = q(t)$, and the first half becomes  
$\dot q(t) = - \nabla  f(q)$.

\end{proof}

Note that if  $J_{0}$ makes  $T_{\partial U}^{*}U$ pseudoconvex, the same holds
 for  
$J_{\lambda }$ the linear interpolation between   $J_{0}$ and  
$J_{1}$.

	From this lemma, and the previous arguments, we may conclude that
 $F H^{*}(\Delta ^{k}, \Gamma ^{\sigma ,k}_{\psi } ; a,b)$ is 
isomorphic to the cohomology of the Thom-Smale-Witten complex of  
$\nabla S_{\Phi }$ restricted to  $S^{b}_{\Phi } - S^{a}_{\Phi }$.

According to \cite{Floer-Witten.complex} this last cohomology equals  
$HI^{*}(S^{b}_{\Phi }, S^{a}_{\Phi }; \nabla S_{\Phi })$.

	Our proof will be complete if we are able to show that
 $$
      HI^{*} (S^{b}_{\Phi } , S^{a}_{\Phi } ; \nabla S_{\Phi }) 
\simeq  HI^{*}(S^{b}_{\Phi }, S^{a}_{\Phi } ; \xi _{\Phi })
 $$
 where  $\xi _{\Phi }$ is as in \cite{V-Loopspaces}. This may be 
explained by looking at figure \ref{fig.12c}.

	We represented there a pseudogradient vector field  $\eta $ for  
$S_{\Phi }$, equal to  $\nabla  S_{\Phi }$ in a neighbourhood of  
${\cal U}_{r,\varepsilon} - {\cal U}_{r,2\varepsilon/3} 
$ and to  $\xi _{\Phi }$ in a neighbourhood of  ${\cal
U}_{r,\varepsilon /2}$.

	Now, all critical points of  $S_{\Phi }$ are inside  ${\cal 
U}_{r,\varepsilon/2}$, and a heteroclinic trajectory for  $\eta 
$ stays inside  ${\cal U}_{r,\varepsilon/2}$ since  $\eta  = 
\xi _{\Phi }$ on  $\partial  {\cal U}_{r,\varepsilon /2}$ is 
tangent to  $\partial  {\cal U}_{r,\varepsilon/2}$.

	Thus  $I^{*}({\cal U}_{r,\varepsilon },\eta ) \simeq  
I^{*}({\cal U}_{r,\varepsilon/2} , \eta )$.

  \begin{figure}
	 \centering
	 \includegraphics*[scale=.8]{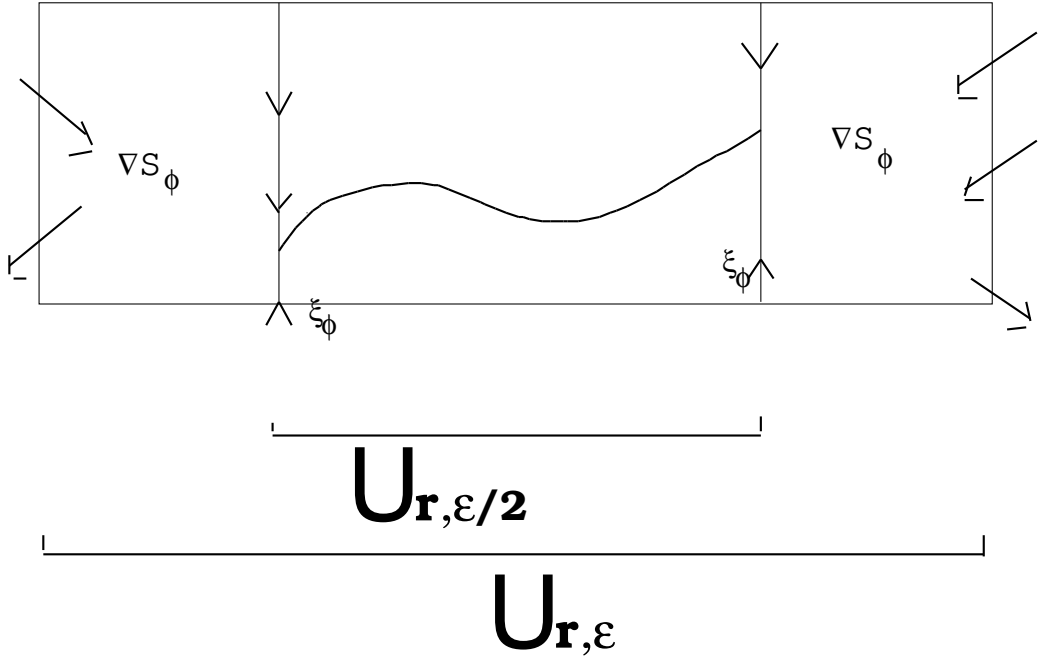}
	\caption{The phase portrait of the pseudo-gradient $\xi_{\Phi}$ of 
	$S_{\Phi}$}
\label{fig.12c}
 \end{figure}

But  $I^{*}({\cal U}_{r,\varepsilon},\eta ) \simeq  
I^{*}({\cal U}_{r,\varepsilon} , \nabla S_{\Phi })$ since  
$\eta  = \nabla  S_{\Phi }$ near  $\partial {\cal U}_{r,\varepsilon}$, and 
 $$
      I^{*} ({\cal U}_{r,\varepsilon/2} , \eta ) \simeq  
I^{*}({\cal U}_{r,\varepsilon/2} , \xi _{\Phi }).
 $$
Therefore
 $$
      I^{*}({\cal U}_{r,\varepsilon } , \nabla S_{\Phi }) 
\simeq  I^{*}({\cal U}_{r,\varepsilon /2} , \xi _{\Phi })
 $$
and our argument obviously  extends if we restrict ourselves to  
$S^{b}_{\Phi }-S^{a}_{\Phi }$. Thus

$$
 I^{*}(S^{b}_{\Phi } , S^{a}_{\Phi } , \nabla S_{\Phi }) \simeq  
I^{*}(S^{b}_{\Phi }, S^{a}_{\Phi }, \xi _{\Phi }).
 $$

Note that these indices do not depend on  $\varepsilon $, for  
$r$ large enough, hence it is not important that the previous inequality 
relates the index of  ${\cal U}_{r,\varepsilon } $ to the 
index of  $U_{r,\varepsilon /2}$.

\section{Appendix}
Let $(M,\omega)$ be a symplectic manifold with contact type boundary,
and $J_0$ be an admissible almost complex structure on $M$.

Given $U$ a domain in $M$, and $x_0$ in $U$, we consider
$$
w_{J_{0},x_{0}} (U)= \inf \lbrace \int_C \omega \quad \mid C\quad \text {is
$J_0$- holomorphic},\quad x_0 \in C \rbrace
$$
and $w(U)$, the usual Gromov width, is then given by
$$
w(U)= \sup \lbrace w_{J_{0}, x_{0}} (U) \mid x_0 \in U,
J_0 \quad \text{is admissible}\quad \rbrace
$$
A natural question is to compute the limits
$$
\displaystyle\lim_{k \to + \infty} w(U^k)=\overline w (U)
$$
$$
\displaystyle\lim_{k \to + \infty} w_{J_{0},x_{0}}(U^k)=\overline
w_{J_{0},x_{0}}(U)
$$
Here we still denote by $J_0$ the almost complex structure $J_0 \times
\ldots \times J_0$ on $M^k$, and by $x_0$ the point $(x_0, \ldots, x_0)$
in $M^k$.

Note that  the sequence $w_{J_{0},x_{0}}(U)$ is obviously
decreasing.

On the other hand, $\overline w (U)$ is not, a priori, equal to sup
$\lbrace \overline w_{J_{0},x_{0}} (U) \vert x_0 \in U, J_0\, \mbox
{is admissible} \rbrace = \widetilde w (U)$ since there are many
more almost complex structures on $M^k$ than those of the type $J_0
\times \ldots \times J_0$. Clearly, we have $\widetilde w (U) \leq
\overline w (U)$.

While it is clear, if $U$ contains a symplectic ball of radius $r$, that
$\overline w(U) \geq \pi r^2$, no such lower bound holds for
$\overline w_{J_{0},x_{0}}(U)$.

In fact, it is not a priori obvious that $\overline w_{J_{0}, x_{0}}(U)$ is
non zero. This is what we prove in this appendix.

\begin{Prop} Let $g_0$ be the metric associated to $(\omega, J_0)$. Then,
if the sectional curvature of $g_0$ is bounded by $c$, and injectivity radius
at $x_0$ bounded by $\rho$, we have
$$
\overline w_{J_0,x_0}(U) \geq \pi \rho^2 \delta(c, \rho)
$$
where $\delta(c, \rho)$ is a continuous positive function, such that $\delta(c, 0)=1$.
\end{Prop}
\begin{Lem}
Let $B_{g_{0}}(r)$ be the ball of radius $r$, centered at $x_0$, for the metric
$g_0$. Let $C$ be a minimal surface for $g_0$ through $x_0$. Then
$$
\mbox{area}\,(C \cap B_{g_{0}}(r)) \geq \pi r^2 \exp(\eta(cr))
$$
where $\eta$ is continuous, $\eta(0)=0$, and $c$ is the upper bound of the
sectional curvature of $g_0$.
\end{Lem}

\proof It is similar to the case where $g_0$ is the euclidean metric.
Set
$$
 a(r)= {\area}\,(C \cap B_{g_0}(r))
$$
then $a'(r)= \mbox{length} (C \cap \partial B_{g_{0}}(r))$ and since $C$ is
minimal, $a(r)$ must be less than the area of the cone through $x_0$,
spanned by $C \cap \partial B_{g_{0}}(r)$.

Let us then compute the area of such a cone. It is clearly given by
$\int_0^r \mbox{length}(\gamma_s) ds$ where
$\gamma_s(t)=\exp (\frac{s}{r}\cdot \exp^{-1}(c_r(t)))$ and
$c_r(t)$ is the curve $C \cap \partial B_{g_{0}}(r)$, the exponential being
taken at $x_0$.

Let $M$ be a bound on the sectional curvature of $g_0$. Then we have, by classical
comparison theorems (\cite{Pansu} p.117, remark 8.14b)
$$
\Vert D \exp_{x_{0}}(u) \Vert \leq \frac{\sinh(Mr)}{Mr}\; {\text for}\; \Vert u
\Vert \leq r
$$
$$
\Vert D \exp_{x_{0}}^{-1} (y) \Vert \leq \frac{Mr}{\sin (Mr)}\; {\text for}\;
y
\in B(x_0,r)
$$
Thus
$$
\mbox{length}\,(\gamma_s) \leq \frac{s}{r} \frac{\sinh (Mr)}{\sin(Mr)} 
\cdot \length(Cr)
$$
and
$$
\int_{0}^{r} \mbox{length}\,(\gamma_s) ds \leq \frac{r}{2}
\frac{\sinh (Mr)}{\sin (Mr)}\cdot \length(Cr)
$$
and we have
$$
a(r) \leq \frac{r}{2} \varphi(Mr)a'(r)
$$
so that 

\begin{eqnarray*}
\frac{a'(r)}{a(r)} & \geq & \frac{2}{r \varphi(Mr)}\\
\log(\frac{a(r)}{a(\varepsilon)}) & \geq &  \log \frac{r^2}{\varepsilon^2} +
\int_\varepsilon^r \frac{2}{ u \varphi(Mu)} (1 - \varphi(Mu)) du\\
&\geq& \log(\frac{r^2}{\varepsilon^2}) + \int_{M \varepsilon}^{Mr} 2
\frac{1-\varphi(v))}{v \varphi(v)} dr
\end{eqnarray*}

Since $1 - \varphi(u) \sim u$ as $u$  goes to zero, the quantity $\int_{M
\varepsilon}^{Mr} \frac{2}{u(\varphi(u))}(1-\varphi(u))du$ converges
to $\eta(Mr)$ as $\varepsilon$ goes to zero, with $\eta$ continuous and
$\eta(0)=0$.

Then, since $\displaystyle\lim_{\varepsilon \to 0}
\frac{a(\varepsilon)}{\varepsilon^2}\geq \pi$, we have $a(r) \geq \pi r^2
\mbox{exp}(\eta(Mr))$. \cfd

Now, replacing $U$ by $U^k$, the sectional curvature of the induced
metrics stays bounded (even though the bound may change as we go from
$k=1$ to $k=2$), and since $U^k$ contains $(B_{g_0}(\rho))^k$, we get
$$
\omega_{J_0, x_0}(U^k) \geq \pi \rho^2 \delta (c, \rho).$$
\cfd

\part*{}

\end{document}